\documentclass[11pt,a4paper]{article}
\usepackage{amsmath,amssymb,amsthm,url}
\usepackage{amsfonts,enumerate}
\usepackage{graphicx}

\usepackage[usenames,dvipsnames]{color}

\input epsf
\input{xy}\xyoption{all}

\usepackage[a4paper,left=2.5cm,right=2.5cm,top=2.5cm,bottom=2.5cm]{geometry}

\newcommand{\jonly}[1]{}
\newcommand{\aronly}[1]{#1}

\def\Z{{\mathbb Z}} \def\R{{\mathbb R}}  
\long\def\comment#1\endcomment{}

\def\sign{\mathop{\fam0 sign}}

\def\diag{\delta}

\theoremstyle{plain}
\newtheorem{theorem}{Theorem}[section]
\newtheorem{lemma}[theorem]{Lemma}
\newtheorem{corollary}[theorem]{Corollary}
\newtheorem{proposition}[theorem]{Proposition}

\theoremstyle{definition}
\newtheorem{remark}[theorem]{Remark}

\begin{document}

\newpage
\title{Stronger counterexamples to the topological Tverberg conjecture}

\author{S. Avvakumov\footnote{University of Copenhagen. Email: \texttt{savvakumov@gmail.com}. Supported by the Austrian Science Fund (FWF), Project P31312-N35 and the European Research Council under the European Union's Seventh Framework Programme ERC Grant agreement ERC StG 716424 -- CASe.},
R. Karasev\footnote{Institute for Information Transmission Problems.
Email: \texttt{r\_n\_karasev@mail.ru}. \texttt{http://www.rkarasev.ru/en/about/}.}
and A. Skopenkov\footnote{Independent University of Moscow, and Moscow Institute of Physics and Technology, and
Email: \texttt{skopenko@mccme.ru}. \texttt{https://users.mccme.ru/skopenko/}.}}

\date{}

\maketitle

\begin{abstract} Denote by $\Delta_M$ the $M$-dimensional simplex.
A map  $f\colon \Delta_M\to\R^d$ is an {\it almost $r$-embedding} if
$f(\sigma_1)\cap\ldots\cap f(\sigma_r)=\emptyset$ whenever $\sigma_1,\ldots,\sigma_r$ are pairwise disjoint faces.
A counterexample to the topological Tverberg conjecture asserts that {\it if $r$ is not a prime power and $d\ge2r+1$, then there is an almost $r$-embedding $\Delta_{(d+1)(r-1)}\to\R^d$}.
This was improved by Blagojevi\'c--Frick--Ziegler using a simple construction of higher-dimensional counterexamples by taking $k$-fold join power of lower-dimensional ones.
We improve this further (for $d$ large compared to $r$): {\it If $r$ is not a prime power and $N=(d+1)r-r\Big\lceil\dfrac{d+2}{r+1}\Big\rceil-2$, then there is an almost $r$-embedding $\Delta_N\to\R^d$}.
The improvement follows from our stronger counterexamples to the $r$-fold van Kampen--Flores conjecture.
Our proof is based on generalizations of the Mabillard--Wagner theorem on construction of almost $r$-embeddings from equivariant maps, and of the \"Ozaydin theorem on existence of equivariant maps.
\end{abstract}



\tableofcontents

\section{Introduction and statement of results}\label{s:mr}

Denote by $\Delta_M$ the $M$-dimensional simplex.
We omit `continuous' for maps.
A map  $f\colon K\to\R^d$ of a union $K$ of closed faces of $\Delta_M$ is an {\bf almost $r$-embedding} if $f(\sigma_1)\cap\ldots\cap f(\sigma_r)=\emptyset$ whenever $\sigma_1,\ldots,\sigma_r$ are pairwise disjoint faces of $K$.
We omit `for any integers $d,r>0$ and $k\ge 0$' at the beginnings of statements.

\medskip
\noindent
{\bf Theorem \ref{t:tvestr}'.}  {\it If $r$ is not a prime power and $d\ge3r+1$, then there is an almost $r$-embedding $\Delta_{(d+1)(r-1)}\to\R^d$.}

\medskip
This is a counterexample to the celebrated topological Tverberg conjecture.
Theorem \ref{t:tvestr}' follows from Theorem \ref{c:any}' of \"Ozaydin and Mabillard--Wagner, together with Lemma \ref{p:redu-ge}' of Gromov--Blagojevi\'c--Frick--Ziegler.
For the history, see the surveys \cite{BBZ, Sk16, BZ16, BS17, Sh18} and the references therein.
In \cite{AMSW} Theorem \ref{t:tvestr}' was improved to $d\ge2r+1$.

The following result gives stronger counterexamples to the topological Tverberg conjecture.

\begin{theorem}\label{t:tvestr} If $r$ is not a prime power and $$N=N(d,r):=(d+1)r-r\Big\lceil\dfrac{d+2}{r+1}\Big\rceil-2,$$
then there is an almost $r$-embedding $\Delta_N\to\R^d$.
\end{theorem}

Theorem \ref{t:tvestr} follows from Theorem \ref{c:any} and Lemma \ref{p:redu-ge}, see the details in \S\ref{s:ded}.

\begin{remark}[motivation and related work]\label{c:motiv}
(a) There naturally appears more general problem:
{\it For which $a,d$ there is an almost $r$-embedding $\Delta_a\to\R^d$?}

This problem was considered in \cite[\S5]{BFZ}, where higher-dimensional counterexamples were constructed from  lower-dimensional ones: {\it If there is an almost $r$-embedding $\Delta_a\to\R^d$, then for each $k$ there is an almost $r$-embedding $\Delta_{k(a+1)-1}\to\R^{k(d+1)-1}$} \cite[Lemma 5.2]{BFZ}.
The proof (exposed a bit simpler \cite[Remark 1.5.c]{Sk16}) is by taking $k$-fold join power as follows.
For two maps $f:\Delta_a\to\R^p$ and $g:\Delta_b\to\R^q$ define the {\it join}
$f*g:\Delta_{a+b+1}=\Delta_a*\Delta_b \to \R^p *\R^q \subset\R^{p+q+1}$ by the formula
$$
(f*g)(\lambda x \oplus (1-\lambda) y):= \lambda f(x)\oplus (1-\lambda) f(y),\quad\text{where}\quad \lambda\in[0,1].
$$
A join of almost $r$-embeddings is an almost $r$-embedding.
Hence the $k$-fold join power of an almost $r$-embedding $\Delta_a\to\R^d$ is an almost $r$-embedding $\Delta_{k(a+1)-1}\to\R^{k(d+1)-1}$.


According to a private communication by F. Frick this procedure \cite[Theorem 5.4]{BFZ} together with the counterexample in \cite[Theorem 1.1]{AMSW} gives an almost $r$-embedding $\Delta_F\to\R^d$ for $r$ not a prime power, $d$ sufficiently large, and $F$ some integer close to $(d+1)r - \dfrac{r+\frac12}{r+1}(d+1)$.
Presumably $F-(d+1)(r-1)$ can be arbitrarily large.

Theorem \ref{t:tvestr} provides even stronger counterexamples to the topological Tverberg conjecture:
{\it for $d$ large compared to $r$ we have $N>(d+1)(r-1)$, and even $N>F$}.
Theorem \ref{t:tvestr} is a partial result on \cite[Conjecture 5.5]{BFZ} stating that
{\it for $r<d$ not a prime power there is an almost $r$-embedding $\Delta_{(d+1)r-2}\to\R^d$
and there are no almost $r$-embeddings $\Delta_{(d+1)r-1}\to\R^d$.}
(The case $r\ge d$ of the conjecture is trivially covered by known results.)
Observe that $N\le dr-2$ for $r<d$.
The second part of the conjecture is addressed in \cite{FS20}.

(b) We think counterexamples of Theorem \ref{t:tvestr} are mostly interesting because their proof requires non-trivial ideas, see Theorems \ref{c:any}, \ref{t:ozaydin-g}, and Lemma \ref{l:zero-map} below.
Thus we do not spell out even stronger counterexamples which presumably could be obtained by combining Theorem \ref{t:tvestr} with the procedure of \cite[\S5]{BFZ} described in (a).
Our proof of Theorem \ref{t:tvestr} is independent of Theorem \ref{t:tvestr}, of \cite{AMSW}, and of the iterated join construction described in (a).

(c) Let us illustrate Theorem \ref{t:tvestr} by numerical examples.
Earlier results gave almost 6-embeddings $\Delta_{280}\to\R^{55}$ and $\Delta_{275}\to\R^{54}$, and, more generally, almost $r$-embeddings $\Delta_{(d+1)(r-1)}\to\R^d$ for $d\ge2r+1$,
\ $\Delta_{d(r-1)}\to\R^{d-1}$ for $d\ge2r+2$, and, even more generally, almost $r$-embeddings  $\Delta_{(d+1-s)(r-1)}\to\R^{d-s}$ for $d\ge2r+s+1$.
Corollary \ref{c:tvestr} below gives an almost 6-embedding $\Delta_{280}\to\R^{54}$, and, more generally, almost $r$-embeddings $\Delta_{(d+1)(r-1)}\to\R^{d-s}$ for certain $r,d,s$.
\end{remark}



\begin{corollary}\label{c:tvestr} Assume that $r$ is not a prime power.

(a) For $q\ge r+2$ and $d=(r+1)q-1$ there is an almost $r$-embedding $\Delta_{(d+1)(r-1)}\to\R^{d-1}$.

(b) If $d\ge(s+2)r^2$ for some integer $s$, then there is an almost $r$-embedding $\Delta_{(d+1)(r-1)}\to\R^{d-s}$.
\end{corollary}

\begin{proof}
Part (a) follows by Theorem \ref{t:tvestr} because $q\ge r+2$, so $((r+1)q-1)r-rq-2\ge(r+1)q(r-1)$.
Part (b) follows by Theorem \ref{t:tvestr} because $d\ge(s+2)r^2\ge(s+1)r^2+r-1$, hence
$$(d+1)(r-1)\le (d-s+1)r-r\dfrac{d-s+2+r}{r+1}-2 \le(d-s+1)r-r\Big\lceil\dfrac{d-s+2}{r+1}\Big\rceil-2.$$
\end{proof}


A {\bf complex} is a collection of closed faces (=simplices) of some simplex.
A {\it $k$-complex} is a complex containing at most $k$-dimensional simplices.
The {\it body} (or geometric realization) $|K|$ of a complex $K$ is the union of simplices of $K$.
Thus continuous or piecewise-linear (PL) maps $|K|\to\R^d$ and continuous maps $|K|\to S^m$ are defined.
We abbreviate $|K|$ to $K$; no confusion should arise.

By general position, any $k$-complex admits an almost $r$-embedding in $\R^{k+\big\lceil\tfrac{k+1}{r-1}\big\rceil}$.
The following counterexample to the $r$-fold van Kampen--Flores conjecture allows to sometimes decrease $\big\lceil\tfrac{k+1}{r-1}\big\rceil$ to $\tfrac{k}{r-1}$ (indeed, take $k=sr$).

\medskip
\noindent
{\bf Theorem \ref{c:any}'.} (\"Ozaydin and Mabillard-Wagner) {\it If $s\ge3$ and $r$ is not a prime power,
then any $s(r-1)$-complex admits an almost $r$-embedding in $\R^{sr}$.}


\medskip
This result follows from the \"Ozaydin
and the Mabillard--Wagner Theorems \ref{t:ozaydin-g}' and \ref{t:mmw}'.
See \cite[\S1, Motivation \& Future Work, 2nd paragraph]{MW14} or the survey
\cite[Theorems 1.7, 3.1 and 3.2, and Remark 1.9.b]{Sk16}.

The following result gives stronger counterexamples to the $r$-fold van Kampen--Flores conjecture.

\begin{theorem}\label{c:any} If $r$ is not a prime power,
then any $k$-complex admits an almost $r$-embedding in $\R^{k+\big\lceil\tfrac{k+3}r\big\rceil}$.
\end{theorem}

Theorem \ref{c:any} is easily deduced below from Theorems \ref{t:ozaydin-g} and \ref{t:mmw}.
The main new ingredient in the proof of Theorems \ref{t:tvestr} and \ref{c:any} is the following Theorem \ref{t:ozaydin-g}.

\smallskip
{\bf Acknowledgments.}
We are grateful to M. Berezovik, F. Frick, A. Magazinov, and the anonymous referees for helpful suggestions.

\section{Deduction of Theorems \ref{t:tvestr} and \ref{c:any} from Theorem \ref{t:ozaydin-g}}\label{s:ded}

\noindent
{\bf Lemma \ref{p:redu-ge}'.} (Constraint) {\it For $M=(sr+2)(r-1)$ if there is an almost $r$-embedding of the union
of $s(r-1)$-faces of $\Delta_M$ in $\R^{sr}$, then there is an almost $r$-embedding $\Delta_M\to\R^{sr+1}$.}

\medskip
Lemma \ref{p:redu-ge}' is due to Gromov \cite[2.9.c]{Gr10} and Blagojevi\'c--Frick--Ziegler \cite[Lemma 4.1.iii and 4.2]{BFZ14}, \cite[proof of Theorem 4]{Fr15'}.
Lemma \ref{p:redu-ge}' has a simple proof (see e.g. the survey \cite[Lemma 1.8]{Sk16}).
The proof  shows that

$\bullet$ $\R^{sr}$ and $\R^{sr+1}$ can be replaced by $\R^{d-1}$ and $\R^d$, respectively.

$\bullet$ $s(r-1)$ and $(sr+2)(r-1)$ can be replaced by $k$ and $(k+2)r-2$, respectively.


\begin{lemma}[Constraint]\label{p:redu-ge} For $M=(k+2)r-2$  if there is an almost $r$-embedding of the union
of $k$-faces of $\Delta_M$ in $\R^{d-1}$, then there is an almost $r$-embedding $\Delta_M\to\R^d$.
\end{lemma}

\begin{proof}[Proof of Theorem \ref{t:tvestr} modulo Theorem \ref{c:any}]
Theorem \ref{t:tvestr} holds for $d=1$ because then $N=r-2$, so $\Delta_N$ does not have $r$ non-empty pairwise disjoint faces.
Assume further that $d\ge2$.
Denote $k:=d-1-\Big\lceil\dfrac{d+2}{r+1}\Big\rceil$.
Then $N=(k+2)r-2$.
Since $d\ge2$ and $r\ge6$, we have $k\ge0$.
We have
$$
\dfrac{d+2}{r+1}=\dfrac{d+2-\frac{d+2}{r+1}}r \ge
\frac{k+3}r
\quad\Rightarrow\quad
d-1=k+\Big\lceil\dfrac{d+2}{r+1}\Big\rceil \ge k+\Big\lceil\dfrac{k+3}r\Big\rceil.
$$
Since $r$ is not a prime power, by Theorem \ref{c:any} there is an almost $r$-embedding
of the union of $k$-faces of $\Delta_N$ to $\R^{d-1}$.
Then by the Constraint Lemma \ref{p:redu-ge} there is an almost $r$-embedding $\Delta_N\to\R^d$.
\end{proof}

Denote by $\Sigma_r$ the permutation group of $r$ elements.
Let $\R^{d\times r}:=(\R^d)^r$ be the set of real $d\times r$-matrices.
The group $\Sigma_r$ acts on $\R^{d\times r}$ by permuting the columns.
Denote
$$
\diag\phantom{}_r=\diag\phantom{}_{r,d}:=\{(x,x,\ldots,x)\in \R^{d\times r}\ |\ x\in\R^d\}.
$$

\noindent
{\bf Theorem \ref{t:ozaydin-g}'.} (\"Ozaydin) {\it If $r$ is not a prime power and $\dim X=d(r-1)$,
then there is a $\Sigma_r$-equivariant map $X\to\R^{d\times r}-\diag\phantom{}_r$.}


\medskip
This is proved in \cite{Oz}, see R. Karasev's short proof in the survey \cite[\S3.2]{Sk16}.

The following result improves the \"Ozaydin Theorem \ref{t:ozaydin-g}'.

\begin{theorem}\label{t:ozaydin-g} If $r$ is not a prime power and $X$ is a complex with a free action of $\Sigma_r$, then there is a $\Sigma_r$-equivariant map $X\to \R^{2\times r}-\diag\phantom{}_r$.
\end{theorem}

\begin{remark}[Relations
to other papers]\label{c:motivkf}
Let $X$ be a complex with a free action of $\Sigma_r$,
Observe that if $\dim X<d(r-1)$, then the existence of an equivariant map $X\to\R^{d\times r}-\diag\phantom{}_r$ follows by general position.
The statements \cite[Theorem 5.1]{AK19}, \cite[Theorem 1.1]{AKu19} of other improvements of Theorem \ref{t:ozaydin-g}'  are obtained from Theorem \ref{t:ozaydin-g} replacing 2 by 1 and imposing stronger restrictions on $r$.\footnote{For a finite cyclic or dihedral group $G$, and a certain representation space $V$ of $G$, $G$-equivariant maps from the classifying space $EG$ to $V-0$ were constructed in \cite{BG17}.
Theorem \ref{t:ozaydin-g} should also be compared to \cite[Theorem 3.6 and the paragraph afterwards]{Ba93}.
That reference takes a group $G$ {\it from a certain class} and proves that there exists \emph{some} representation $W$ of $G$, for which there exist $G$-equivariant maps $X \to S(W)$ for certain $G$-spaces $X$.
However, $G=\Sigma_r$ does not belong to that class, and the $\Sigma_r$-space $S(W)$ described in
\cite[Theorem 3.6 and the paragraph afterwards]{Ba93} need not coincide with the $\Sigma_r$-space
$\R^{2\times r}-\diag\phantom{}_r$ given by Theorem \ref{t:ozaydin-g}.}

Our proof of Theorem \ref{t:ozaydin-g} is analogous to the argument in \cite{AK19, AKu19}: {\it Theorem \ref{t:ozaydin-g} follows from the known Lemma \ref{lemma:zero-degree} and the new Lemma \ref{l:zero-map} below} (see also the paragraph after Lemma \ref{l:zero-map}).
This is different from the \"Ozaydin idea \cite{Oz}
and from the short proof in \cite[\S3.2]{Sk16}.
So our argument gives a simple proof of the \"Ozaydin Theorem \ref{t:ozaydin-g}'.
\end{remark}

For a complex $K$ let $K^{\times r}_{\Delta}$ be the associated $r$-fold deleted product:
$$K^{\times r}_{\Delta}:= \bigcup \{ \sigma_1 \times \cdots \times \sigma_r
\ : \sigma_i \textrm{ a simplex of }K,\ \sigma_i \cap \sigma_j = \emptyset \mbox{ for every $i \neq j$} \}.$$
The group $\Sigma_r$ has a natural action on the set $K^{\times r}_{\Delta}$,
permuting the points in an $r$-tuple $(p_1,\ldots, p_r)$.
This action is evidently free and {\it PL}, i.e. compatible with some structure of a complex on
$K^{\times r}_{\Delta}$.


\medskip
\noindent
{\bf Theorem \ref{t:mmw}'.} (Mabillard-Wagner) {\it Assume that $K$ is a $s(r-1)$-complex and $s\ge3$.
There exist an almost $r$-embedding $K\to\R^{sr}$ if and if there is a $\Sigma_r$-equivariant map
$K^{\times r}_{\Delta}\to\R^{sr\times r}-\diag\phantom{}_r$.}



\begin{theorem}[Mabillard-Wagner]\label{t:mmw}
Assume that $K$ is a $k$-complex and $rd\ge(r+1)k+3$.
There exists an almost $r$-embedding $f:K\to\R^d$ if and only if there exists a $\Sigma_r$-equivariant map
$K^{\times r}_{\Delta}\to \R^{d\times r}-\diag\phantom{}_r$.
\end{theorem}

See the proofs in \cite{MW15},
\cite[\S3]{Sk16}, and in \cite{MW16, MW16', Sk17}, respectively.\aronly{\footnote{For a criticism of the proof of Theorem \ref{t:mmw} in \cite{MW16, MW16'} see \cite[\S5]{Sk17}; this footnote is not present in the published version of this paper.}}



\begin{proof}[Proof of Theorem \ref{c:any} modulo Theorems \ref{t:ozaydin-g} and \ref{t:mmw}]
Let $K$ be any $k$-complex and $d:=k+\Big\lceil\dfrac{k+3}r\Big\rceil$.
If $d=1$, then $k=0$, so Theorem \ref{c:any} is obvious.
Now assume that $d\ge2$.
Since $r$ is not a prime power, by Theorem \ref{t:ozaydin-g} there is a $\Sigma_r$-equivariant map
$K^{\times r}_\Delta\to\R^{2\times r}-\diag\phantom{}_r$.
The composition of this map with the $r$-th power of the inclusion $\R^2\to\R^d$ gives
a $\Sigma_r$-equivariant map $K^{\times r}_\Delta\to \R^{d\times r}-\diag\phantom{}_r$.
We have $rd\ge(r+1)k+3$.
Hence by Theorem \ref{t:mmw}  there is an almost $r$-embedding $K\to\R^d$.
\end{proof}



\section{Proof of Theorem \ref{t:ozaydin-g}}

\begin{lemma}\label{lemma:zero-degree} Let $G$ be a finite group acting on $S^n$.
If there exists a degree zero $G$-equivariant self-map of $S^n$,
then any complex $X$ with a free action of $G$ has a $G$-equivariant map $X\to S^n$.
\end{lemma}

See the historical remarks and a proof in \cite[\S5]{AK19}.
In particular, this lemma follows from \cite[Lemma 3.9]{Ba93}; see \cite[\S5]{AK19} for a simpler
direct proof.\aronly{\footnote{Note that to read the direct proof in \cite[\S5]{AK19} is simpler than to find the notation required for the statement \cite[Lemma 3.9]{Ba93} and deduce Lemma \ref{lemma:zero-degree} from that statement; this footnote is not present in the published version of this paper.}}

Denote by $S^{d(r-1)-1}_{\Sigma_r}\subset\R^{d\times r}-\delta_r$ the set formed by all $d\times r$-matrices in which the sum of the elements in each row is zero, and the sum of the squares of all the matrix elements is 1.
This set is invariant under the action of $\Sigma_r$.
This set is homeomorphic to $S^{d(r-1)-1}$.

\begin{lemma}\label{l:zero-map} If $r$ is not a prime power, then there is a degree zero $\Sigma_r$-equivariant self-map of $S:=S^{2r-3}_{\Sigma_r}=S^{2(r-1)-1}_{\Sigma_r}$.
\end{lemma}

Lemma \ref{l:zero-map} is analogous to \cite[Theorem 4.2]{AK19} and \cite[Theorem 1.4.c,d]{AKu19}.
Those theorems are stated in a different language, but can be obtained from Lemma \ref{l:zero-map} by replacing $2r-3$ by $r-2$, and adding stronger restrictions on $r$.
The proofs follow the same plan via Proposition \ref{p:spech}
(although this proposition is not explicitly stated in \cite{AK19, AKu19}).
The binomial coefficients appear in the same way.
However, the procedure of obtaining the prescribed sign in front of the binomial coefficient
requires additional work.
The procedure is easier in \cite{AK19}, is intermediate here, and is more complicated in \cite{AKu19}
(the proof of \cite{AKu19} also uses additional ideas).

\begin{proof}[Proof of Lemma \ref{l:zero-map}]
Since $r$ is not a prime power, the greatest common divisor of the binomial coefficients
$\binom{r}{k}$, $k=1,\ldots,r-1$ is $1$ \cite{lucas1878}.
Hence $-1$ is an  integer linear  combination of the binomial coefficients.
Denote by $C\subset S$
the set of $(2\times r)$-matrices whose second row is zero, and the entries of the first row involve only two numbers.
A {\it special} map is a $\Sigma_r$-equivariant self-map $f$ of $S$
which is a local homeomorphism in some neighborhood of $C$.
The identity map of $S$
is a special map of degree 1.
Thus the lemma is implied by the following assertion

{\it For any $r$, any $k=1,\ldots,r-1$, and any special map $f$
there are special maps $f_+,f_-$ such that $\deg f_\pm=\deg f\pm\binom{r}{k}$.}

This is implied by the following proposition.
\end{proof}

\begin{proposition}\label{p:spech} For any $r$, any $k=1,\ldots,r-1$, and any special map $f$ there are a point $c\in C$ and $\Sigma_r$-equivariant homotopies $h_+,h_-:S\times I\to \R^{2r-2}_{\Sigma_r}$ such that

$(1_\pm)$ $h_{\pm,0}=f$, and $h_{\pm,1}:S\to S$ is special,


$(2_\pm)$ $h$ is a local homeomorphism over a neighborhood of 0,   and
\[
\deg h_{\pm,1}-\deg f = \deg_0 h = \pm \binom{r}{k}\sign\phantom{}_c f.
\]
Here $\deg_0 h$ is the degree of $h$ over 0, and $\sign_c f\in\{+1,-1\}$ is the sign of the preimage $c$ of $f(c)$ under the map $f$ (since $c\in C$ and $f$ is special, $f$ is a homeomorphism in a neighborhood of $c$; we have $\sign_c f=\sign\det df(c)$ if $f$ is smooth and $\det df(c)\ne0$).
\end{proposition}

Informally, we construct $h_{\pm}$ by `pushing' a certain point $c\in C$ and its orbit towards the origin in
$\R^{2\times r}$.
See \cite[Figures 1 and 2]{AK19}.
For $h_+$ such a pushing is `twisted along the reflection with respect to a certain hyperplane'.

\begin{proof}[Proof of Proposition \ref{p:spech}]
{\it Definitions of $c,G,U$ and $\rho$.}
The objects we construct depend on $r,k$ but we suppress $r,k$ from their notation.
Define the vector
$$
M:=(\underbrace{k-r,\ldots,k-r}_{k},\underbrace{k,\ldots,k}_{r-k}).
$$
Define the $(2\times r)$-matrix
$c:=\left(\begin{matrix} M/|M| \\0\end{matrix}\right)\in C$.
The orbit $\Sigma_rc$ of $c$ contains $\binom{r}{k}$ points.
The stabilizer of $c$ is $G:=\Sigma_k\times\Sigma_{r-k}\subset\Sigma_r$.

The standard metric on the sphere is $\Sigma_r$-invariant.
Hence there is a small ball $U$ centered at $c$ such that
$U\cap\sigma U=\emptyset$ for any $\sigma\in\Sigma_r-G$, and $\sigma U = U$ for any $\sigma \in G$.
Take a smooth function $\rho':S\to [0,1]$ which is zero outside $U$, and is one in a neighborhood of $c$.
Define a smooth function $\rho :S\to [0,1]$ by $\rho(x):=\frac{1}{r!}\sum_{\sigma\in\Sigma_r} \rho'(\sigma x)$.
Then $\rho$ is zero outside $\Sigma_rU$, is one in a neighborhood of $\Sigma_rc$, and is invariant with respect to the $\Sigma_r$-action.

\smallskip
{\it Construction of $h_-$.}
For $t\in[0,1/2]$ define
\[
h_-(x,t)=h_{-,t}(x) := \begin{cases} f(x) & x\not\in\Sigma_rU\\
f(x) - 4t\rho(x)f(\sigma c) & x\in\sigma U,\ \sigma\in\Sigma_r.
\end{cases}
\]
Clearly, $h_-$ is well-defined,  is continuous, and is $\Sigma_r$-equivariant.


In this paragraph we  prove that
$$h_-^{-1}(0)=(\Sigma_rc)\times1/4$$
for the constructed homotopy $h_-:S\times[0,1/2]\to \R^{2r-2}_{\Sigma_r}$.
If $t\in[0,1/2]$ and $h_{-,t}(x)=0$, then $x\in\Sigma_rU$.
Since $f$ is a local homeomorphism and $|f(x)|=1$, we have $4t\rho(x)=1$ and $x=\sigma c$ for some $\sigma\in\Sigma_r$.
Then $t=1/4$.

Since $h_-^{-1}(0)=(\Sigma_rc)\times1/4$, we have $0\not\in h_{-,1/2}(S)$.
Hence there is a homotopy\footnote{E.g. take $h_{-,t}(x) = \frac{h_{-,1/2}(x)}{2-2t+(2t-1)|h_{-,1/2}(x)|}$.} $h_-:S\times[1/2,1]\to \R^{2r-2}_{\Sigma_r}-\{0\}$ between $h_{-,1/2}$ and a map
$h_{-,1}$ defined by $h_{-,1}(x):=\frac{h_{-,1/2}(x)}{|h_{-,1/2}(x)|}$.
Then $h_-^{-1}(0)=(\Sigma_rc)\times1/4$
for the constructed homotopy $h_-:S\times[0,1]\to \R^{2r-2}_{\Sigma_r}$.

\smallskip
{\it Proof of $(1_-)$.}
Clearly, $h_{-,0}=f$.
Since $f$ is a local homeomorphism in some neighborhood of $C$,
the map $h_{-,1}$ is such in a neighborhood of $C-\Sigma_r c$.
In a neighborhood of $\sigma c$ the map $h_{-,1}$ is a shift by $-2f(\sigma c)$ composed with the central projection back to the sphere.
This is clearly a homeomorphism in a neighborhood of $\sigma c$.

\smallskip
{\it Proof of ($2_-$).}
Take a sufficiently small neighborhood $W$ of $c$ such that $\rho(W)=1$
and $f(W)$ is contained in the hemisphere centered at $f(c)$.
Then $h_{-,t}(x) = f(x)-4t f(c)$ for any $x\in W$ and $t\in[0,1/2]$.
Let $f(c)^\perp$ be the hyperplane tangent to $S$ at $f(c)$.
Let $\pi : \R^{2r-2}_{\Sigma_r}\to f(c)^\perp$ be the orthogonal projection.
Let $\left<f(c)\right>$ be the line passing through $f(c)$ and the origin (and so orthogonal to $f(c)^\perp$).
Then $N := \pi^{-1}(\pi(f(W)))$ is a neighborhood of this line.
Observe that $\pi|_{f(W)}$ is a homeomorphism.
Denote by $\pi^{-1}$ its inverse.
Define a map
\[
\tau :  N \to N\quad\text{by}\quad \tau(y) := y + \pi(y) - \pi^{-1}(\pi(y)).
\]
Then $\tau$ shifts every line in $N$ parallel to $\left<f(c)\right>$ by the vector
$\pi(y) - \pi^{-1}(\pi(y))\in \left<f(c)\right>$.
Hence $\tau$ is a self-homeomorphism of $N$ preserving the orientation.
The origin is fixed under $\tau$ because $\tau(0) = 0 + \pi(0) - \pi^{-1}(\pi(0)) = c - c = 0$.
For $t\in[0,1/2]$ we have
\[
\tau(h_{-,t}(x)) = \tau(f(x) - 4t f(c)) = f(x) - 4tf(c) + \pi(f(x)) - \pi^{-1}(\pi(f(x))) = \pi(f(x)) - 4t f(c).
\]
Hence for the decomposition $\R^{2r-2}_{\Sigma_r}=\left<f(c)\right>\times f(c)^\perp$ the
map $\tau\circ h_-|_{W\times[0,1/2]}$ is the Cartesian product of the maps
$$\pi\circ f : W\to f(c)^\perp\quad\text{and}\quad  a : [0,1/2]\to \left<f(c)\right>,\quad\text{where}\quad a(t):=-4tf(c).$$
Both $\pi\circ f$ and $a$ are embeddings.
Hence $\tau\circ h_-|_{W\times[0,1/2]}$ is a homeomorphisms in a neighborhood of $(c,1/4)$.
Hence $h_-|_{W\times[0,1/2]}$ and $h_-|_{W\times[0,1]}$ are also such.
Since $h_-$ is $\Sigma_r$-equivariant and $h_-^{-1}(0)=(\Sigma_rc)\times1/4$, we see that
$h_-$ is a local homeomorphism over a neighborhood of 0.

Denote $D := \partial(S\times[0,1]) = S\times\{0,1\}$.
Then
\[
\deg h_{-,1}-\deg f = \deg h_{-,1}- \deg h_{-,0} = \deg(h_-|_D:D \to S) = \deg_0 h_-.
\]
We also have
\begin{multline*}
\frac{\deg_0 h_-}{\binom{r}{k}} =  \sign\phantom{}_{(c,1/4)} h_- =
\sign\phantom{}_0 \tau^{-1} \cdot \sign\phantom{}_{(c,1/4)}(\tau\circ h_-) =
\sign\phantom{}_{(c,1/4)}(\tau\circ h_-) = \\
= \sign\phantom{}_{1/4} a\cdot \sign\phantom{}_c (\pi\circ f) =
- \sign\phantom{}_{f(c)} \pi\cdot \sign\phantom{}_c f = - \sign\phantom{}_c f.
\end{multline*}


\smallskip
{\it Construction of $h_+$.}
Define the $(2\times r)$-matrix
$c_1:=\left(\begin{matrix} 0\\ M/|M| \end{matrix}\right)\in S$.
Take the hyperplane $c_1^\perp\subset\R^{2r-2}_{\Sigma_r}$ orthogonal to $c_1$ and passing through the origin.
Then $c\in c_1^\perp$.
We may assume that $V := U\cap\rho^{-1}[1/3,1]$ is a ball by assuming that $\rho$ is radially symmetric in $U$.
Let $q:V\to V$ [$g$] be the restriction to $V$ of the reflection with respect to the hyperplane $c_1^\perp$.
Then $q$ is $G$-equivariant, $\sign_cq=-1$ and $q^{-1}(c)=c$.

Define a
$G$-equivariant map $\phi' : U\times\{0\} \cup (\partial U\cup V)\times[0,1] \to U$

$\bullet$ on $U\times \{0\}$ as the natural homeomorphism;

$\bullet$ on $\partial U\times[0,1]$ as the composition of the projection and the inclusion $\partial U\to U$;

$\bullet$ on $V\times[1/3,1]$ as the composition of the projection, $q$, and the inclusion $V\to U$;

$\bullet$ on $V\times[0,1/3]$ as a $G$-equivariant homotopy between the identity map $\phi'_0$ and $q=\phi'_{1/3}$.


By the Borsuk homotopy extension theorem \cite[\S5.5]{FF89} $\phi'$ extends to a homotopy
$\psi:U\times[0,1]\to U$.
Define a homotopy $\phi : U\times[0,1]\to U$ by considering the average of $\psi$ with respect to $G$:
\[
\phi(x,t) = \frac{1}{|G|}\sum_{g\in G} g \psi(g^{-1}x,t) \in U.
\]
We have $\phi=\phi'=\psi$ on $U\times\{0\}\cup(\partial U\cup V)\times[0,1]$ because $\phi'$ is $G$-equivariant.
The homotopy $\phi$ is $G$-equivariant, since for any $m\in G$ from the linearity of the action of $G$ on $U$ one obtains
\[
m\phi(x,t) = \frac{1}{|G|}\sum_{g\in G} mg \psi(g^{-1}x,t) = \frac{1}{|G|}\sum_{k=mg\in G} k \psi(k^{-1} m x,t) = \phi(mx,t).
\]
Extend $\phi$ to $\Sigma_rU\times[0,1]$ in a $\Sigma_r$-equivariant way.
Define for $t\in[0,1/2]$
\[
h_+(x,t)=h_{+,t}(x) := \begin{cases} f(x) & x\not\in\Sigma_rU\\
f(\phi(x,2t)) - 4t\rho(x)f(\sigma c) & x\in\sigma U,\ \sigma\in\Sigma_r.
\end{cases}
\]
Clearly, $h_+$ is well-defined (since $\phi$ is $G$-equivariant), is continuous, and is $\Sigma_r$-equivariant.

In this paragraph we prove that
\[
h_+^{-1}(0)=(\Sigma_rc)\times1/4
\]
for the constructed homotopy $h_+:S\times[0,1/2]\to \R^{2r-2}_{\Sigma_r}$.
If $h_t(x)=0$ and $t\in[0,1/2]$, then $x\in\Sigma_rU$.
Since $f$ is a local homeomorphism and $|f(x)|=1$, we have $4t\rho(x)=1$ and $\phi(x,2t)=\sigma c$ for some $\sigma\in\Sigma_r$.
Therefore $\rho(x)\ge\frac14$, $t\ge 1/4$, and $x\in\sigma V$.
Hence $\phi(x,2t)=\sigma q(\sigma^{-1}x)$.
Since $q^{-1}(c)=c$, we have $x=\sigma c$, $\rho(x) = 1$, and $t = 1/4$.

Then analogously to the construction of $h_-$ we extend $h_+$ to $S\times[0,1]$, and check the properties $(1_+)$ and $(2_+)$.
In the proof of $(2_+)$ we have $\deg_0 h_+ = -{r\choose k}\sign_c q\sign_c f = {r\choose k}\sign_c f$
Here $\sign_cq$ appears because for $t\ge 1/6$ and $x\in V$ we have $\phi(x, 2t)=q(x)$.
\end{proof}


{\it In this list books, surveys and expository papers are marked by stars}

\end{document}